\documentclass[reqno, 12pt]{amsart}

\usepackage{caption}
\usepackage{array}
\usepackage{amsmath}
\usepackage{amsfonts}
\usepackage{amssymb}
\usepackage{mathtools}

\usepackage{enumerate}
\usepackage{epsfig}
\usepackage{graphicx}
\usepackage{amsthm}
\usepackage{amsmath, amscd}
\usepackage{xy}
\xyoption{all}
\usepackage{color}
\usepackage{marvosym}
\usepackage{bm}
\usepackage{amsbsy}
\usepackage{tikz}
\usetikzlibrary{matrix,arrows,backgrounds}%, decorations, pathmorphing, positioning, fit}

\newtheorem{theorem}{Theorem}[section]
\newtheorem{theorem-section}{Theorem}[section]

\newtheorem{proposition}[theorem]{Proposition}
\newtheorem{proposition-section}[theorem-section]{Proposition}
\newtheorem{corollary}[theorem]{Corollary}
\newtheorem{corollary-section}[theorem-section]{Corollary}

\newtheorem{question}[theorem]{Question}

\theoremstyle{definition}
\newtheorem{definition}[theorem]{Definition}
\newtheorem{remark}[theorem]{Remark}
\newtheorem{remark-section}[theorem-section]{Remark}

\newtheorem{example}[theorem]{Example}

\newcommand{\op}[1]{\operatorname{#1}}

\newcommand{\widebar}[1]{\overline{#1}}

\def\NN{\op{\mathbb{N}}}

\def\RR{\op{\mathbb{R}}}

\def\FF{\op{\mathbb{F}}}

\def\AA{\op{\mathbb{A}}}
\def\PP{\op{\mathbb{P}}}

\def\P{\op{P}}
\def\NP{\op{NP}}

\def\Paff{\op{P}_{\op{aff}}}
\def\NPaff{\op{NP}_{\op{aff}}}
\def\Psch{\op{P}_{\op{sch}}}
\def\NPsch{\op{NP}_{\op{sch}}}

\def\Pproj{\op{P}_{\op{proj}}}
\def\NPproj{\op{NP}_{\op{proj}}}
\def\VNP{\op{VNP}}
\def\VP{\op{VP}}

\def\bigO{\op{O}}
\def\UCR{\op{R}}
\def\Res{\op{Res}}

\def\loccit{\emph{loc. cit. }}

\def\Res{\op{Res}}

\def\nn{\underline{{\mathbf{n}}}}
\def\mm{\underline{{\mathbf{m}}}}

\def\on{\op{n}}
\def\om{\op{m}}

\def\AA{\op{\mathbb{A}}}

\def\cC{\mathcal{C}}

\newcommand{\defemph}[1]{\emph{#1}}

\title{Complexity Classes and Completeness in Algebraic Geometry}
\author{M. Umut Isik}
\address{
  \begin{tabular}{l}
   Mehmet Umut Isik  \\ 
   \hspace{.1in} Department of Mathematics, University of California, Irvine \\
   \hspace{.1in} Email: { isik@math.uci.edu } \\
  \end{tabular}
}

\numberwithin{equation}{section}

\begin{document}
\begin{abstract} We study the computational complexity of sequences of projective varieties. We define analogues of the complexity classes P and NP for these and prove the NP-completeness of a sequence called the universal circuit resultant. This is the first family of compact spaces shown to be NP-complete in a geometric setting.  
\end{abstract}

\maketitle

Valiant's theory of algebraic/arithmetic complexity classes is an algebraic analogue of Boolean complexity theory, where Boolean functions are replaced by polynomials over any ring, and Boolean circuits are replaced by arithmetic circuits that use the $\times$, $+$ operations of the ring instead of the Boolean operations. An interesting point about Valiant's theory is how sequences in $\VNP$ are made. Keeping with the Boolean procedure of taking an efficiently computable `verifier' $g_n(x_1,\dots,x_n)$ and searching through all `proofs' $e$ it might accept by taking 
    $$ f_n(\cdot) = \bigvee_{e\in \{0,1\}^{m(n)}} g_{n+m(n)}(\cdot,e), $$
    Valiant's theory defines the sequences in $\VNP$ as those of the form 
$$ f_n(\cdot) = \sum_{e\in \{0,1\}^{m(n)}} g_{n+m(n)}(\cdot,e), $$  
where $g_n(x_1,\dots,x_n) \in \FF\left[ x_1,\dots,x_n \right]$ form an efficiently computable sequence. In this sense, Valiant's theory, while a source of interesting geometry problems, is a purely algebraic analogue of Boolean complexity.  

What would a geometric analogue look like? If we wanted to make a naive, geometric version of Boolean complexity, we would do the following. Take sequences of $X_n \subset \FF_{2}^{n}$, and say they are in $\P$ if they are the zero-set of polynomials $\FF_{2}[x_1,\dots,x_n]$ which can be computed by arithmetic circuits of polynomially bounded size. Then, define $\NP$ by taking sequences $(Y_n \subset \FF_2^{n})$ in $\P$ and projecting them down to the first $n$ coordinates: $X_n = \pi(Y_{n+m(n)})$. It is easy to see that we can make arithmetic circuits over $\FF_2$ efficiently simulate Boolean circuits and vice-versa; so the these two theories are the same. 

The issue arises when we generalize from a finite field to an arbitrary field $\FF$. For general fields, the projections of affine varieties are not necessarily closed subsets. One solution is to generalize from algebraic subsets to semi-algebraic subsets, which would give a non-uniform version of the Blum-Shub-Smale (BSS) theory, which is a geometric analogue of the Turing theory. Such an approach would be a more basic version of the complexity of constructible functions and sheaves studied in \cite{basuconstr}.  

The alternative, if one wants to stay within algebraic geometry, is to work with projective varieties. This is the approach we take in this paper. The basic objects we consider are sequences $(X_n)_{n=1}^{\infty}$ of projective (and bi-projective) varieties. A sequence $(X_n)$ is in $\Pproj$ if, for each $n$, there are homogeneous polynomials whose zero-set is $X_n$, which are the outputs of a homogeneous arithmetic circuit of polynomially bounded size, and which have polynomially bounded degree. Sequences in $\NPproj$ are the projections of sequences of $X_n \subset \PP^{\on_1(n)}\times \PP^{\on_2(n)}$ in $\Pproj$, onto the first component $\PP^{\on_{1}(n)}$. This is similar to the compact projective complexity classes studied in the BSS setting in \cite{basutoda}; the main difference is that we are in the arithmetic circuit model of computation.  

Key concepts in computational complexity are those of reduction and completeness. A sequence $(X_n \subset \PP^{\on(n)})$ reduces to $(Y_n\subset \PP^{\om(n)})$ if there are linear maps $\rho_{n}: \PP^{\on(n)} \to \PP^{\om(p(n))}$, where $p(n)$ is a polynomially bounded function, such that $\rho^{-1}(Y_n) = X_n$. Our main result is that there is a sequence, which we call the \emph{universal circuit resultant} which is $\NPproj$-complete (or, more precisely, the Segre embedding of the universal circuit resultant, c.f. Theorems~\ref{thmucrishard} and~\ref{thmsegreucrcomplete}). A key tool in the construction of the universal circuit resultant is a modified version of the universal circuits of \cite{raz}. 

The universal circuit resultant is the first example of a compact family proven to be $\NP$-complete in a geometric complexity setting. The $\NP$-completeness of the usual resultant (the homogeneous Nullstellensatz problem) is a long-standing open problem in BSS theory \cite{shubproblems}.  

We also discuss similar definitions for affine varieties and projective schemes. In the affine case, we define $\NPaff$ by taking closures of projections of families in $\Paff$; this way, the image under a projection stays affine. However, it is not known whether $\NPaff$ is closed under reductions. In the scheme case, the separation of $\Psch$ and $\NPsch$ should be an easier problem than in the projective variety case.

Why study a complexity theory of algebraic varieties? One motivation is to make an analogue of the $\P$ vs $\NP$ problem, in algebraic geometry, that might possibly be easier to make progress on. The second aim is to lay the ground-work for notions and questions of computational complexity theory to be used as tools in algebraic geometry. Many complexity-related results on geometry questions involve starting with finite field, or rational coefficients so that complexity can be discussed with the Turing machine or Boolean models. While this makes perfect sense for finite field questions and for concrete applications; from an algebraic geometry point of view, it adds the additional difficulty of arithmetic geometry on top of the difficulty of algebraic geometry. 
The language of this paper would be useful for expressing
results on the complexity of de-singularization, 
minimal-model program, projective duality and other interesting geometric operations without assuming rational coefficients or any bit reduction.

While we discuss a geometric theory of complexity, there is no direct relation between this work and the Geometric Complexity Theory program of Mulmuley and Sohoni (\cite{gct1}, \cite{gctsurvey} and references therein). GCT is an approach to proving lower bounds in algebraic complexity using deep statements in representation theory. As such, it can also prove lower bounds on the complexity of hypersurfaces, but the theory discussed here is not advanced enough to ask such questions at this point.    

\subsection*{Acknowledgments} I would like to thank Vladimir Baranovsky and Saugata Basu for useful discussions on the subject of this paper. 

\vspace{0.1in}

\section{Complexity of Projective Varieties}

We refer to the books \cite{Buergetal, BuergisserCompleteness} for background on algebraic/arithmetic complexity. 

\subsection{Arithmetic circuits}
We use arithmetic circuits as our model of computation. Let $\FF$ be any field.
An arithmetic circuit over $\FF$ with input variables $x_1,\dots,x_n$ is a labeled directed acyclic multigraph (two or more distinct edges from the same source to the same target are allowed). Every node with in-degree $0$ is called an input-gate and is labeled by an input variable or a constant in $\FF$. Every other node is labeled by $\times$ or $+$, making it a product-gate or sum gate. Every edge is labeled by an element of $\FF$ called the weight of the edge. Nodes with no outgoing edges are called output gates. 

Each gate computes a polynomial in $\FF\left[ x_1,\dots,x_n \right]$ as follows. Input gates calculate the variable or constant they are labeled with. If $v$ is a sum-gate, and its incoming edges are with weights $\lambda_1,\dots,\lambda_d$ from gates computing $f_1,\dots,f_d$, then $v$ computes the polynomial $\sum_{i=1}^{d}\lambda_i f_i$. If $v$ is a product gate, it computes $\prod_{i=1}^{d}\lambda_i f_i$. The size of a circuit is the number of edges in the circuit. 

An equivalent model is that of a straight-line program. These are sequences of instructions that start with a list of the input variables, and each subsequent instruction is weighted sum or product of earlier instructions.

\subsection{Basic definitions}
The basic objects we want to consider are sequences of projective varieties $(X_n) = (X_n)_{n=1}^{\infty}$. But, since we will be taking projections, we want to also consider sequences of bi-projective varieties, i.e.\ sequences $X_n \subset \PP^{\on_1(n)}\times \PP^{\on_2(n)}$ of varieties, where $\on_i: \NN \to \NN$ are any polynomially bounded functions. By abuse of notation, we will write $\nn(n)= (\on_1(n),\on_2(n))$ and 
$$X_n \subset\PP^{\nn(n)} = \PP^{\on_1(n)}\times\PP^{\on_2(n)}.$$
Sequences of projective varieties are the special case where $\on_2(n) = 0$ identically. 

In order to prevent our language from becoming too cumbersome, we may drop the `bi' in bi-projective varieties and bi-homogeneous polynomials; referring to them as projective and homogeneous when the meaning is clear from the context. 

\begin{definition}
  A homogeneous algebraic circuit (or straight-line program) $C$  \defemph{computes} a projective variety $X\subset \PP^{n_1}\times \PP^{n_2}$ if its output polynomials $f_1,\dots,f_m$ are homogeneous polynomials, and  $X$ is the zero-set of $f_1,\dots,f_m$ in $\PP^{n_1}\times \PP^{n_2}$. 
\end{definition}

In other words, $C$ computes $X$ if it can solve the problem of determining whether a point in $\PP^{n_1}\times \PP^{n_2}$ is in $X$. The cost of the computation is the size of the algebraic circuit.

\begin{definition}
  The \emph{complexity} $\op{C}(X)$ of $X \subset \PP^{n_1}\times \PP^{n_2}$ is the size of the smallest homogeneous algebraic circuit that computes $X$. For a sequence $(X_n)$, the complexity is a function $\NN \to \NN$ that takes each $X_n$ to its complexity. This function is also denoted by $\op{C}(X_n)$ by abuse of notation. 
\end{definition}

\begin{example}
  As a basic example, let $X_n = \left\{ \left[ 1:0:\dots:0 \right] \right\} \subset \PP^{n}$. Let $x_0,\dots,x_n$ be the homogeneous variables for $\PP^{n}$. Then, a circuit of size $n$, which directly outputs $x_1,\dots,x_n$ computes $X_n$ so $\op{C}(X_n)\leq n$. Each of the variables $x_1,\dots,x_n$ must be involved in any computation of $\left\{ [1:0:\dots:0] \right\}$ because otherwise the zero-set would be a cone. So the above is the most efficient computation of $X_n$ and $\op{C}(X_n) = n$.
\end{example}

\begin{example}
  Consider the zero-set $X\subset \PP^{n}$ of a polynomial $f$ in $\FF[x_0,\dots,x_n]$. Then the complexity of $X$ is bounded above by the algebraic complexity of $f$; but the two may differ for two reasons. First, since we are looking at the complexity of the zero-set rather than scheme-theoretic complexity, we could be looking at a non-reduced polynomial, $f = g^r$. Second, if our field is not algebraically closed, we wouldn't have Nullstellensatz, so we could have, for example, $\FF=\RR$ and $f = x_0^2 + \dots + x_n^2$, and the zero-set would be empty, making the arithmetic complexity of $f$ different from the complexity of its zero-set.  
\end{example}

\begin{example}\label{exuniversalquadric}
  [Universal Quadric] Consider the universal quadratic equation $$q_n(x_0,\dots,x_n,a_1,\dots,a_N) = a_0 x_0^2 + a_1 x_0 x_1 + \dots + a_N x_n^2$$
  in variables $x_0,\dots,x_n$. Here, $q_n$ has a term for each degree 2 monomial, so $N = {n+2 \choose 2}-1$. The zero-set $Q_n \subset \PP^{n}\times \PP^{N}$ is called the \emph{universal quadric}. By making a circuit that computes each term and adds the result, we see that the universal quadric has complexity in $\bigO(n^2)$. Similarly, the solution to the universal degree $d$ polynomial has complexity in $\bigO(n^{d})$.   
\end{example}

\subsection{Complexity classes, reduction and completeness}

\begin{definition}
  A sequence $(X_n\subset \PP^{\nn})$ is in $\Pproj$ if there is a sequence of homogeneous circuits, of size polynomially bounded in $n$, that compute polynomials $f_1,\dots,f_{m(n)}$ of degree and number polynomially bounded in $n$, whose zero-set is $X_n$.  
\end{definition}

To define $\NPproj$, we consider the sequences of projections  
$$\pi_{n}^{\on_1,\on_2} : \PP^{\on_1}\times \PP^{\on_2} \to \PP^{\on_1},$$
where $\on_1$ and $\on_2$ are functions polynomially bounded in $n$.

\begin{definition}
  A sequence $(X_n \subset \PP^{\on_1})$ is in $\NPproj$ if there is a sequence $(Y_n \subset \PP^{\on_1} \times \PP^{\on_2})$ such that $(Y_n)$ is in $\Pproj$ and $X_n = \pi_{n}^{\on_1,\on_2}(Y_n)$. 
\end{definition}

We now discuss a sequence which is in $\NPproj$. 

\begin{example}[Resultant]\label{exresultant}  
Let $V_n = \FF^{n+1}$ be a vector space of dimension $n+1$. Fix a degree $d$ and consider 
$$\PP^{\on(n)} = \PP(S^dV_n\times \dots \times S^d V_n) $$
where there are $n+1$ terms in the product, making $\on(n)={n+d \choose d}(n+1)-1$.
Every point in $\PP^{\on(n)}$ corresponds to a tuple of $n+1$ equations of degree $d$ in $n+1$ variables. 

Define $\Res_{d,n}\subset \PP^{\on (n)}$ to be the subvariety consisting of the points that correspond to tuples of equations, homogeneous of degree $d$, that have common non-zero solutions; i.e., $\Res_{d,n}$ is the resultant. In other words, letting $Z_{d,n}$ be the projectivization of $\op{ev}^{-1}(0)$, where $\op{ev}:\AA^{\on(n)+1} \times \AA^{n+1} \to \AA^{n+1}$ is the evaluation map; $\Res_{d,n}$ is the image $\Res_{d,n} = \pi_1 (Z_{d,n})$.  
Since $d$ is fixed, $Z_{d,n}$ is in $\Pproj$, and therefore we have that $\Res_n$ is in $\NPproj$. There is also the related example of the discriminant variety which we do not discuss here.
\end{example}

We can now discuss reductions and completeness. 
In Boolean complexity and in Valiant complexity \cite{ValiantClasses}, reductions are usually taken to be substitutions of variables and constants. Asymptotically, there is no significant difference between considering substitution reductions and affine-linear reductions, c.f. \cite{kayalaffine}. 

For families of projective varieties, we take reductions to be linear maps; and this is all we need for Theorem~\ref{thmsegreucrcomplete}. But when we consider products of projective spaces, we must be more careful. To this end, let us define a \emph{projective-linear map} $\PP^{n_1}\times \dots \times \PP^{n_s} \to \PP^{m_1}\times \dots \times \PP^{m_r}$ to be a map that is mapping, to each component, as either a linear map depending on only one of the $\PP^{n_i}$ or a constant map. An example of a projective-linear map $\PP^{1}\times \PP^{1} \to \PP^{1} \times \PP^{1} \times \PP^{1}$ is the map $([x_0:x_1],[y_0:y_1])\to ([y_0+2 y_1:y_1],[3 : 5] ,[ 4 x_0:x_0+x_1])$. Note that the target spaces we will consider will only have two components. 
Of course, other reductions, such as reductions linear in each set of variables but higher in total degree could also be considered, c.f. Remark~\ref{remmultiprojectivesegre}. 

\begin{definition}
  We say that a sequence $(X_n \subset \PP^{\nn})$ reduces to a sequence $(Y_n \subset \PP^{\mm})$ if there is a polynomially bounded function $p:\NN \to \NN$ and a sequence of projective-linear maps $\rho_{n} : \PP^{\nn(n)} \to \PP^{\mm (p(n))}$ such that, for each $n$, $X_n$ is the pullback of $Y_{\mm(p(n))}$ via $\rho_{n}$, i.e.\ $\rho_{n}^{-1}(Y_{p(n)}) = X_n$.
\end{definition}

Reductions are sometimes called $p$-projections in the literature. We reserve the word projection for taking the image under geometric projections $\pi: A \times B \to A$.  

\begin{remark}
  Unlike the affine case (c.f. section \ref{affinesection}), $\Pproj$ and $\NPproj$ are automatically closed under polynomial reductions as the image of a projective variety is always closed. 
\end{remark}

\definition A sequence $(X_n)$ is said to be $\cC$-hard for a complexity class $\cC$ if every sequence in $\cC$ can be reduced to $(X_n)$. If $(X_n)$ is $\cC$-hard and is in $\cC$, then $(X_n)$ is said to be $\cC$-complete.

\subsection{Universal circuits} 
We now describe a modified version of the construction, in \cite{raz}, of a $\VP$-complete sequence constructed via \emph{universal circuits} (see also \cite{durandetal, mahajansaurabh} for other $\VP$-complete sequences). The difference is that we want to handle the bi-homogeneous case. We will later use this sequence of circuits to make $\NPproj$-complete and $\Pproj$-complete sequences. 
It should be noted that this section could have been based also on the `genetic computations' studied in \cite{BuergisserCompleteness} Chapter 5.

The universal circuits of \cite{raz} are based on the ability to put any circuit into normal homogeneous form.  

\begin{definition}[Normal bi-homogeneous form]
  A bi-homogeneous arithmetic circuit is in normal bi-homogeneous form if it satisfies the following:
  \vspace{-0.1in}
  \begin{enumerate}[(i)]
    \item All leaves are labeled by input variables.
    \item All edges from the leaves go to sum gates.
    \item All output gates are sum gates.
    \item Gates are alternating in the sense that if $(u,v)$ is an edge and if $v$ is a sum gate, then $u$ is a product gate and vice versa.
    \item The fan-in of every product gate is 2.
    \item The fan-out of every sum gate is 1. 
  \end{enumerate}
  \vspace{0.1in}
\end{definition}

\begin{proposition}\label{propnormalform}
For every bi-homogeneous circuit $C$ of size $s$, there is a circuit $C'$ in normal bi-homogeneous form, of size $\bigO(s)$, that has the same output.  
\end{proposition}
\begin{proof}
  We refer to \cite[Proposition 2.3]{raz} for the proof of the fact that every circuit of size $s$ can be turned into circuit in normal homogeneous form. We simply observe that all the steps in the proof in \loccit preserve the bi-homogeneousness of the circuit. One difference between the statement in \loccit and this one is that size $\bigO(s)$ is enough for us; whereas, \loccit has the size as $\bigO(sr^2)$, where $r$ is the maximum degree of the output polynomials. The difference is that \loccit starts with non-homogeneous circuits, and the $r^2$ factor is the cost of homogenizing the circuit. So, as we already start with bi-homogeneous circuits, the normal form circuit has size $\bigO(s)$. 
\end{proof}

We say that a circuit $\Phi$ is \emph{universal} for a set $S$ of circuits if, for every $C\in S$, there is an assignment of multipliers for the edges of $\Phi$ that makes $\Phi$ compute the same polynomials as $C$. 

The proposition \cite[Proposition 2.8]{raz} proves the existence of universal circuits; we adapted it to the bi-homogeneous case we are considering as follows. 

\begin{theorem}
  \label{thmuniversalcircuit}
For every $r_1$, $r_2$, $n$, $m$ and $s\in \NN$, with $s\geq n+m$, there exists a circuit $\Phi$ that is universal for the set of circuits of size $s$ that are bi-homogeneous in $n$ and $m$ inputs, and that output $n$ polynomials of bi-degree $(r_1,r_2)$. Moreover, $\Phi$ is in normal bi-homogeneous form and has $\bigO(r_1^2r_2^2s)$ gates.   
\end{theorem}
\begin{proof} 
  Replace $s$ by a constant factor of $s$ to account for the cost of normalization of a circuit of size $s$ (Proposition~\ref{propnormalform}).

  We describe the gates of the universal circuit $\Phi$ in groups which are the inputs, the sum-levels and the product-levels. There are:
  \begin{itemize}
    \item[-] $n+m$ input gates: $n$ inputs for variables of bi-degree $(1,0)$ and $m$ inputs for variables of bi-degree $(0,1)$.
    \item[-] $r_1r_2$ sum-levels, indexed by pairs of natural numbers $(l_1,l_2)$. Each sum-level contains $r_1r_2s$ sum-gates and no product gates. 
    \item[-] $r_1r_2$ product-levels, again indexed by pairs of natural numbers $(l_1,l_2)$, each containing $r_1r_2s$ product gates and no sum gates.  
  \end{itemize}
A gate in sum-level $(l_1,l_2)$ (respectively, product-level $(l_1,l_2)$), will compute a bi-homogeneous polynomial of bi-degree $(l_1,l_2)$. Let us now describe the edges of $\Phi$:
\begin{itemize}
  \item[-] The inputs of each sum gate at sum-level $(l_1,l_2)$ are all the product gates in product-level $(l_1,l_2)$. So, we imagine the sum-level $(l_1,l_2)$ as being above the product level $(l_1,l_2)$. There are no product-levels $(1,0)$ and $(0,1)$. Each sum-gate at sum-level $(1,0)$ is connected to every input corresponding to the variables of degree $(1,0)$; similarly for $(0,1)$. Each sum gate has fan-out at most 1. 
  \item[-] In product-level $(l_1,l_2)$, each gate has fan-in equal to 2. For each pair $(j_1,j_2)$ with $1\leq j_1 \leq l_1$, $1\leq j_2 \leq l_2$, there are $s$ product gates which take their inputs from sum gates in sum-level $(j_1,j_2)$ and $(l_1-j_1, l_2,-j_2)$. These gates are called product-level $(l_1,l_2)$ gates of \emph{type} $(j_1,j_2)$. This is why we took $r_1r_2s$ gates in each level; each sum gate has fan-out only one, so there need to be enough sum gates for the product gates to connect to. Since all sum gates in a sum-level $(j_1,j_2)$ are connected to all product gates in the product-level $(j_1,j_2)$, it does not matter which specific sum gate is connected to a product gate in product-level $(l_1,l_2)$.
\end{itemize}

This completes the description of $\Phi$. Note that, according to the above description, there are some gates which are left with no outgoing connections. These can be removed without effecting the result.  
Let $C$ be a bi-homogeneous circuit of size $s$, computing $n$ polynomials of bi-degree $(r_1,r_2)$ on $n+m$ variables. By Proposition~\ref{propnormalform} above, $C$  can be replaced by a circuit $C'$ in normal bi-homogeneous form. Being in normal bi-homogeneous form, the gates of $C'$ can be grouped into the same levels and types as above. So we can embed $C'$ into $\Phi$ and set the unused edges in $\Phi$ to $0$, provided there are enough gates, which is true since we made $\Phi$ have $s$ product gates for each product-level and type, and more than $s$ gates for each sum-level. 
\end{proof}

We can relax the degree requirement by putting together $r^2$ versions of $\Phi$; one for each pair of degrees. 

\begin{corollary}\label{coralldegreeuniversal}
  For every $r$, $n$, $m$ and $s\in \NN$, with $s\geq r(n+m)$, there exists a circuit $\Phi_{n,m,r,s}$ with $\bigO(r^6s)$ gates that is universal for the set of bi-homogeneous circuits of size $s$ that have $n+m$ inputs, and that output, for each $r_1,r_2 \leq r$, at most $n$ polynomials which are bi-homogeneous of bi-degrees $(r_1,r_2)$. 
\end{corollary}

\begin{remark}\label{remmultihomog}
  We focused in this section on the bi-homogeneous case, but all the results of this section have obvious generalizations to the multi-homogeneous case with fixed number of components. The $l$-homogeneous $\Phi$ and $\Phi'$ would have $n_1+\dots+n_l$ inputs, and there would be $r^l$ sum-levels and $r^{l}$ product levels. But, $l$ would have to be fixed in sequences because otherwise the possible multi-degrees would become exponentially many. 
\end{remark}

\subsection{The Universal Circuit Resultant}
The aim of this section is to make $\Pproj$-complete and $\NPproj$-complete sequences. 

Let $\Phi_{n,m,r,s}$ be a universal circuit produced in Corollary~\ref{coralldegreeuniversal} above. We want to modify $\Phi_{n,m,r,s}$ so that the outputs of the circuits that we embedded in $\Phi_{n,m,r,s}$ become reductions of the outputs of the modified circuit $\Phi_{n,m,r,s}'$. For this, we add, as input to the circuit, a control-variable $t_e$ for each edge $e$ in $\Phi_{n,m,r,s}$, break the edge into two and put a multiplication gate to multiply with $t_e$ in the middle. All the edges of $\Phi_{n,m,r,s}'$ have multiplier one, but the computed polynomials are now controlled by the $t$-variables.  

A few observations: 
\begin{itemize}
  \item[-] Since each gate in $\Phi_{n,m,r,s}$ has fan-in at most $r^2s$, $\Phi_{n,m,r,s}'$ has $\bigO(r^8s^2)$ control variables $t_e$. Let $N=N(n,m,r,s)$ denote the number of control variables. 
  \item[-] $\Phi_{n,m,r,s}'$ has $\bigO(r^8s^2)$ gates. 
  \item[-] The outputs of $\Phi'_{n,m,r,s}$ are tri-homogeneous polynomials of degree at most $r$ (component-wise) in the original variables and degree $2r$ in the new control variables $t_e$. 
\end{itemize}

As the outputs of $\Phi'_{n,m,r,s}$ are tri-homogeneous, they have a well-defined zero-set $Z_{n,m,r,s} \subset \PP^{n-1}\times \PP^{m-1} \times \PP^{N-1}$. We project it using the projection map 
$$\pi_{1,3}:\PP^{n-1}\times \PP^{m-1}\times \PP^{N-1} \to \PP^{n-1} \times \PP^{N-1}.$$ 

\begin{definition}
  The \emph{universal circuit resultant}, $\UCR_{n,m,r,s} \subset \PP^{n}$, is the image $\pi_{1,3}(Z_{n,m,r,s})$ of the zero-set of the outputs of $\Phi_{n,m,r,s}'$. 
\end{definition}

To get a sequence, we need to specialize the circuit size and the degree, as will be apparent later, it is better to not specialize $m$. We set $r=n$ and $s=n^2+n+m$. Here, the choice of $n^2+n+m$ is arbitrary, in the sense that any $s$ whose difference to $n^2$, the maximum number of outputs of $\Phi_{n,m,r,s}'$ when $r=n$, is a non-constant polynomial in $n$ would have sufficed. We added $m$ here just to make sure that $s\geq n+m$. For simplicity, let us abuse notation and write
$$\UCR_{n,m} := \UCR_{n,m,n,n^2+n+m} \subset \PP^{n-1}\times \PP^{N(n)-1},$$
and $\Phi'_{n,m} := \Phi'_{n,m,n,n^2+n+m}$, and 
and its zero-set,
$Z_{n,m} = Z_{n,m,n,n^2+n+m}$ as necessary.

The universal circuit resultant $\UCR_{n,m}$ is the family that will be shown to the $\NPproj$-complete; but it has two indices and we defined completeness for only single-indexed families. But, $(\UCR_{n,m})$ can easily be re-indexed to have only one index that enumerates all pairs, with only quadratic loss.

\begin{theorem}\label{thmucrishard}
  $(\UCR_{n,m})$ is $\NPproj$-hard (in the sense that its single-index re-indexing is $\NPproj$-hard).  
\end{theorem}
\begin{proof}
  Let $X_n = \pi_1(Y_n)$, be a sequence in $\NPproj$; with $Y_n \subset \PP^{\on_1(n)-1}\times \PP^{\on_2(n)-1}$ a sequence in $\Pproj$. 

  We want to show that $(X_n)$ reduces to $(\UCR_{n,m})$. Let $C_n$ be a bi-homogeneous circuit of polynomially bounded size $s(n)$ which computes $Y_n \subset \PP^{\on_1}\times \PP^{\on_2}$. Let the maximum component of the multi-degree of the outputs of $C_n$ be $d(n)$. For each $n$, let $q(n)$ be the maximum of $\on_1(n)+\on_2(n)+1$, $d(n)$ and $s(n)$. Then, by Corollary~\ref{coralldegreeuniversal} and the discussion above, the control variables $t_i$ or each edge in $\Phi'_{q,\on_2,q,q^2+q+\on_2}$ can be set to appropriate values $\tau_i$ so that the non-zero outputs of $C_n$ and $\Phi'_{q,\on_2,q,q^2+q+\on_2}$ are identical. 
  Therefore, for all $x_1,\dots,x_{n_1},y_1,\dots,y_{\on_2}$, we have, after omitting extra zero-outputs, 
  \begin{multline*}
  C_n(x_1,\dots,x_{n_1},y_1,\dots,y_{\on_2}) = \Phi'_{q,\on_2,q,q^2+q+\on_2}(x_1,\dots,x_{n_1},0,\dots,0, \\ y_1,\dots,y_{\on_2},\tau_1,\dots,\tau_N).
  \end{multline*}
  Hence, for $\left[ x_1:\dots:x_{\on_1} \right]\in \PP^{n_1}$, there exists $\left[ y_1:\dots:y_{\on_2} \right]\in \PP^{\on_2}$ such that 
$$C_n(x_1,\dots,x_{n_1},y_1,\dots,y_{\on_2})=(0,\dots,0),$$ 
if and only if there exists $\left[ y_1:\dots:y_{\on_2} \right]\in \PP^{\on_2}$ such that $$\Phi'_{q,\on_2,q,q^2+q+\on_2}(x_1,\dots,x_{n_1},0,\dots,0, \\ y_1,\dots,y_{\on_2},\tau_1,\dots,\tau_N)=(0,\dots,0).$$

So, if make the reduction $\rho_{n}: \PP^{\on_1-1} \to \PP^{q-1} \times \PP^{N(n,m)-1}$ by mapping
  $$\rho([x_1,\dots,x_{n_1}]  = ([x_1,\dots,x_{n_1},0,0,\dots,0],[\tau_1,\dots,\tau_N]),$$ 
then, we have 
$\rho^{-1}\UCR_{n,m} = X_n$. 
Thus, ($X_n$) reduces to ($\UCR_{n,m})$.
\end{proof}

While $(\UCR_{n,m} \subset \PP^{n-1}\times \PP^{N(n,m)-1})$ is $\NPproj$-hard, it is not $\NPproj$-complete because sequences in $\NPproj$ were defined to be sequences of varieties in projective spaces, not products of projective spaces. While this is not a fundamental issue and could be remedied by slightly changing the definitions, c.f. Remark~\ref{remmultiprojectivesegre}, it is also possible to use the Segre embedding to make, directly, a complete sequence of projective varieties as follows. 

Consider  the image $\sigma_{n,m}\UCR_{n,m} \subset \PP^{nN(n,m)-1}$ of $\UCR_{n,m}$ under the Segre embedding 
$$\sigma_{n,m} :  \PP^{n-1}\times \PP^{N-1} \to \PP^{nN-1},$$ 
$$\sigma_{n,m}(\left[ x_i \right]_{i}, [t_j]_{j}) = [x_{i}t_j]_{i,j}.$$

\begin{theorem}\label{thmsegreucrcomplete}
  $(\sigma_{n,m}\UCR_{n,m})$ is $\NPproj$-complete. 
\end{theorem}

\begin{proof}
  Consider the incidence variety $Z_{n,m} \subset \PP^{n-1}\times\PP^{m-1}\times \PP^{N(n,m)-1}$ whose projection $\pi_{1,3}(Z_{n,m}) = \UCR_{n,m}$. The equations for $Z_{n,m}$ are computed by the universal circuit $\Phi_{n,m}'$. 
  If we consider the induced map, 
  $$\widetilde{\sigma_{n,m}} :  \PP^{n-1}\times \PP^{m-1} \times \PP^{N(n,m)-1} \to \PP^{nN(n,m)-1}\times \PP^{m-1},$$
  then we have $\pi_1(\widetilde{\sigma_{n,m}}(Z_{n,m})) = \sigma_{n,m}(\UCR_{n,m})$. So, to see that $\sigma_{n,m}\UCR_{n,m}$ is in $\NPproj$, it suffices to show that $(\widetilde{\sigma_{n,m}}(Z_{n,m}))$ is in $\Pproj$. 

  To compute $\widetilde{\sigma_{n,m}}(Z_{n,m})$, consider the equations $f_1,\dots,f_{M}$ of $Z_{n,m}$,which are the outputs of $\Phi_{n,m}'$. We have $M = n^4$, since the number of top level sum-gates in $\Phi_{n,m}'$. These are tri-homogeneous in the $n$-many $x$-variable inputs, $m$-many $y$-variable inputs of $\Phi_{n}'$ and the control variables $t_i$. Consider the homogeneous coordinates of $\PP^{nN-1}\times \PP^{m-1}$ as the variables $\left\{ z_{ij} \right\}$ and $\left\{ y_i \right\}$. 

  We make the equations for $\widetilde{\sigma_{n,m}}(Z_{n,m})$ in two sets:
\begin{itemize}
    \item[-] $\bigO(n^2)$ quadratic equations in the $z$-variables for the image of the Segre embedding.
    \item[-] $\bigO(n^5)$ equations to cut out $\widetilde{\sigma_{n,m}}(Z_{_n,m})$ inside $\widetilde{\sigma_{n,m}}(\PP^{n-1}\times \PP^{N-1})$. 
\end{itemize}
To make the second set of equations, we will turn $f_1,\dots,f_{n^2}$ into equations in the $z$-variables. Since $z_{ij}=x_it_j$ under the Segre map, we want to ensure that the $x$-degree and the $t$-degree of each $f_i$ are equal. But since the $t$-variables are control variables for the edges in $\Phi_{n,m}$, each $x$-variable is already multiplied by a $t$-variable in the first set of edges of the circuit. This means that the variable $z_{ij}$ can already be substituted in the circuit for $x_it_j$ to remove all $x$ variables from the circuit. After this process, by the construction of $\Phi'_{n,m}$, there are still $t$-variables in each $f_i$. These can also be replaced by $z$-variables by taking, $n$ copies of the circuit, and replacing, in the $j$th copy, every occurrence of every $t_i$ by $z_{ij} = t_ix_j$. This adds, for each equation $f_i$, equations of the form $x_0^{u}f_i, x_1^{u}f_i,\dots,x_n^{u}f_i$. All-together, these equations cut out $\widetilde{\sigma_{n,m}}(Z_{n,m})$, and the final circuit produced still has polynomial size. This completes the proof that $(\sigma_{n,m}\UCR_{n,m})$ is in $\NPproj$.

To see that $\sigma_{n,m}\UCR_{n,m}$ is also $\NPproj$-hard. Recall that, in the proof of the $\NPproj$-hardness of $\UCR_{n,m}$ above (Theorem~\ref{thmucrishard}), the reduction maps $$\rho_{n}: \PP^{\on_1-1} \to \PP^{q-1} \times \PP^{N-1}$$ were of the form: 
  $$\rho([x_1,\dots,x_{\on_1}]  = ([x_1,\dots,x_{\on_1},0,0,\dots,0],[\tau_1,\dots,\tau_N]),$$ 
  where $\tau_1,\dots,\tau_{N}$ were constants. For the reduction to $\sigma_{n,m}\UCR_{n,m}$, we take the reduction $\rho'_{n,m}(x_1,\dots,x_{\on_1}) = (\tau_jx_i)_{i,j}$, with $x_i = 0$ for $i>\on_1$. 
\end{proof}

\begin{remark}
    The above proof gives the following general fact. If 
    $$(X_n\subset \PP^{\on_1(n)}\times \PP^{\on_2(n)})$$ is $\NPproj$-complete, with projective-linear reductions which are non-constant on only one of the components, then the Segre embedding of $(X_n)$ is also $\NPproj$-complete. 
\end{remark}

There is also a $\Pproj$-complete family. Recall that the universal circuit $\Phi'_{n,0,n,n^2+n}$ is a homogeneous circuit, with $n$ so-called $x$-variables, no $y$-variables and $N = N(n,0)$ control variables $t_i$. 

\begin{proposition}\label{proppcomplete}
   The sequence of zero-sets of $\Phi'_{n,0,n,n^2+n}$ in $\PP^{n-1} \times \PP^{N-1}$ is $\Pproj$-complete.    
\end{proposition}
\begin{proof}
Since $\Phi'_{n,0,n,n^2+n}$ can simulate any homogeneous circuit of size less than or equal to $n^2+n$ (Corollary~\ref{coralldegreeuniversal}), we immediately have that sequences of projective varieties in $\Pproj$ reduce to the sequence of zero-sets of $\Phi'_{n,0,n,n^2+n}$. To get the reductions for bi-projective varieties, we consider Segre embeddings as we did above.
\end{proof}

\begin{remark}
  \label{remmultiprojectivesegre} 
  We could also have chosen to work with multi-homogeneous varieties $(X_n\subset \PP^{\on_1(n)}\times \dots \times \PP^{\on_{l(n)}(n)})$, defined complexity in the same way using multi-homogeneous circuits and multi-homogeneous varieties.
  If $l(n)$ was required to be bounded by a constant for all $n$, then the above theorems would still hold with $r^l$ levels in the proof of Theorem~\ref{thmuniversalcircuit} and larger, but polynomially growing circuits everywhere, with $l$ in the exponent. But, for Theorem~\ref{thmucrishard}, we would need to change the reductions to multi-linear reductions; this is so that the Segre map itself can be considered a reduction.  
  If $l(n)$ was not required to be constant, but required to be polynomially bounded, then the proof Theorem~\ref{thmuniversalcircuit}, and therefore of subsequent theorems would not work as the number of `levels' and therefore the size of the universal circuits would increase exponentially in $n$.  
\end{remark}

\section{Alternative versions}

\subsection{Affine Varieties}\label{affinesection} There is an alternative to going directly to projective varieties. We can consider, as basic objects, sequences $(X_n \subset \AA^{\on_1(n)})$. Such sequences are in $\Paff$ if there are polynomials computed by circuits of polynomially bounded size that cut out $X_n$. 

We define the class $\NPaff$ as those sequences of the form $\widebar{\pi(Y_{n+m(n)})}$ for sequences $(Y_n)$ in $\Paff$, where $\pi: \AA^{n+m(n)} \to \AA^{n}$ is the projection onto the first $n$ coordinates and $m$ is a polynomially bounded function. The closure ensures that the projection stays affine. This is consistent with algebraic elimination as well, since if we take an ideal $I\subset k\left[ x_1,\dots,x_n,y_1,\dots,y_m \right]$, then the elimination ideal $I\cap k\left[ x_1,\dots,x_n \right]$ is the ideal of the closure of the projection of $V(I)\subset \AA^{n+m}$.

The problem is that it is not clear whether this class $\NPaff$ is closed under reductions. More precisely, we are asking the following question:

\begin{question}
 Given $f_1,\dots,f_k$, in variables $x_1,\dots,x_n$, computed by an arithmetic circuit of size $s$, can one construct polynomials $g_1,\dots,g_s$ in variables $x_1,\dots,x_n,$ and $z_1,\dots,z_N,t$, computed by an arithmetic circuit of size polynomially bounded in $s$, such that the image of the zero-set of $g_1,\dots,g_s$ under a projection map is the cone of the projective closure of the zero-set of $f_1,\dots,f_k$? 
\end{question}

In other words, does the operation of taking projective closure `reduce' to geometric elimination? 

This question is also the main technical challenge in proving that the resultant (of quadratics) is $\NPproj$-complete (or also $\NP$-complete in BSS theory).   

\subsection{Scheme-Theoretic Complexity} Now consider, as basic objects sequences of closed subschemes $(X_n \to \PP^{\on_1(n)})$ of projective space and closed subschemes $(X_n \to \PP^{\on_1(n)}\times \PP^{\on_2(n)})$ of bi-projective space.  We say that a sequence $(C_n)$ of circuits computes $(X_n)$ if for each $n$, $X_n$ is the zero-scheme of the outputs of $C_n$. This gives a definition of $\Psch$.

We define $\NPsch$ just we defined $\NPproj$, but using the scheme-theoretic image under the projection. A sequence $(X_n \subset \PP^{\on_1(n)})$ reduces to $(Y_n\subset \PP^{\om(n)})$ if there are linear maps $\rho_{n}: \PP^{\on_1(n)} \to \PP^{\om(p(n))}$, where $p(n)$ is a polynomially bounded function, such that the fiber product $$Y_n \times_{\PP^{m(n)}} P^{\on_1(n)} = X_n.$$

The proof of Theorem~\ref{thmucrishard} carries over exactly to the scheme case; but the proof of Theorem~\ref{thmsegreucrcomplete} does not.\ 

\nocite{GKZ}
\nocite{ValiantPermanent, basuzell}

%\bibliography{complexity}

\begin{thebibliography}{DMM{\etalchar{+}}14}

\bibitem[Bas12]{basutoda}
Saugata Basu, \emph{A complex analogue of {T}oda's theorem}, Foundations of
  Computational Mathematics \textbf{12} (2012), no.~3, 327--362.

\bibitem[Bas15]{basuconstr}
\bysame, \emph{A complexity theory of constructible functions and sheaves},
  Foundations of Computational Mathematics \textbf{15} (2015), no.~1, 199--279.

\bibitem[BCS13]{Buergetal}
Peter B{\"u}rgisser, Michael Clausen, and Amin Shokrollahi, \emph{Algebraic
  complexity theory}, vol. 315, Springer Science \& Business Media, 2013.

\bibitem[B{\"u}r13]{BuergisserCompleteness}
Peter B{\"u}rgisser, \emph{Completeness and reduction in algebraic complexity
  theory}, vol.~7, Springer Science \& Business Media, 2013.

\bibitem[BZ10]{basuzell}
Saugata Basu and Thierry Zell, \emph{Polynomial hierarchy, {B}etti numbers, and
  a real analogue of {T}oda's theorem}, Foundations of Computational
  Mathematics \textbf{10} (2010), no.~4, 429--454.

\bibitem[DMM{\etalchar{+}}14]{durandetal}
Arnaud Durand, Meena Mahajan, Guillaume Malod, Nicolas de~Rugy-Altherre, and
  Nitin Saurabh, \emph{Homomorphism polynomials complete for {VP}},
  LIPIcs-Leibniz International Proceedings in Informatics, vol.~29, Schloss
  Dagstuhl-Leibniz-Zentrum fuer Informatik, 2014.

\bibitem[GKZ08]{GKZ}
Israel~M Gelfand, Mikhail Kapranov, and Andrei Zelevinsky, \emph{Discriminants,
  resultants, and multidimensional determinants}, Springer Science \& Business
  Media, 2008.

\bibitem[Kay12]{kayalaffine}
Neeraj Kayal, \emph{Affine projections of polynomials}, Proceedings of the
  forty-fourth annual ACM symposium on Theory of computing, ACM, 2012,
  pp.~643--662.

\bibitem[MS01]{gct1}
Ketan~D Mulmuley and Milind Sohoni, \emph{Geometric complexity theory i: An
  approach to the {P} vs. {NP} and related problems}, SIAM Journal on Computing
  \textbf{31} (2001), no.~2, 496--526.

\bibitem[MS16]{mahajansaurabh}
Meena Mahajan and Nitin Saurabh, \emph{Some complete and intermediate
  polynomials in algebraic complexity theory}, International Computer Science
  Symposium in Russia, Springer, 2016, pp.~251--265.

\bibitem[Mul11]{gctsurvey}
Ketan~D Mulmuley, \emph{On {P} vs. {NP} and geometric complexity theory:
  Dedicated to {S}ri {R}amakrishna}, Journal of the ACM (JACM) \textbf{58}
  (2011), no.~2, 5.

\bibitem[Raz08]{raz}
Ran Raz, \emph{Elusive functions and lower bounds for arithmetic circuits},
  Proceedings of the fortieth annual ACM symposium on Theory of computing, ACM,
  2008, pp.~711--720.

\bibitem[Shu14]{shubproblems}
Michael Shub, \emph{Some problems for this century}, Talk presented at
  FOCM’14, December 11, 2014, Montevideo, 2014.

\bibitem[Val79a]{ValiantClasses}
Leslie~G Valiant, \emph{Completeness classes in algebra}, Proceedings of the
  eleventh annual ACM symposium on Theory of computing, ACM, 1979,
  pp.~249--261.

\bibitem[Val79b]{ValiantPermanent}
\bysame, \emph{The complexity of computing the permanent}, Theoretical computer
  science \textbf{8} (1979), no.~2, 189--201.

\end{thebibliography}
\bibliographystyle{amsalpha}
\newcommand{\etalchar}[1]{$^{#1}$}
\def\cprime{$'$}
\providecommand{\bysame}{\leavevmode\hbox to3em{\hrulefill}\thinspace}
\providecommand{\MR}{\relax\ifhmode\unskip\space\fi MR }
\providecommand{\MRhref}[2]{%
  \href{http://www.ams.org/mathscinet-getitem?mr=#1}{#2}
}
\providecommand{\href}[2]{#2}

\end{document}